\documentclass[11pt]{amsart} 
\usepackage{amsmath,amssymb,a4wide,ulem} 
\usepackage{mathabx}
\usepackage{color}
\usepackage{xcolor,graphicx}
\usepackage{mathrsfs}
\usepackage{tikz,pgfplots}
\usepackage{subcaption}

\usepackage{scalerel,stackengine}
\stackMath
\newcommand\reallywidehat[1]{%
\savestack{\tmpbox}{\stretchto{%
  \scaleto{%
    \scalerel*[\widthof{\ensuremath{#1}}]{\kern.1pt\mathchar"0362\kern.1pt}%
    {\rule{0ex}{\textheight}}
  }{\textheight}%
}{2.4ex}}%
\stackon[-6.9pt]{#1}{\tmpbox}%
}
\newtheorem{theorem}{Theorem} 
\newtheorem{lemma}[theorem]{Lemma} 
\newtheorem{proposition}[theorem]{Proposition} 
\newtheorem{remark}[theorem]{Remark} 
\newtheorem{corollary}[theorem]{Corollary}

\newcommand{\R}{{\mathbb R}} 
\newcommand{\C}{{\mathbb C}} 
\newcommand{\N}{{\mathbb N}}

\renewcommand{\H}{{\mathbb H}}
\newcommand{\E}{{\mathbb E}}

\newcommand{\dd}{{\rm d}}
\newcommand{\ii}{\ensuremath{\mathrm{i}}}
\newcommand{\e}{\mathrm{e}}
\DeclareMathOperator{\ReT}{Re}
\DeclareMathOperator{\ImT}{Im}

\newcommand{\abs}[1]{\left\vert#1\right\vert}

\providecommand{\norm}[1]{\left\lVert#1\right\rVert}

\usepackage{color}

\title[]{Strong rates of convergence of a splitting scheme for Schr\"{o}dinger equations with 
nonlocal interaction cubic nonlinearity and white noise dispersion 
}

\date{\today}

 \author{Charles-Edouard Br\'ehier}
 \address{Univ Lyon, Université Claude Bernard Lyon 1, CNRS UMR 5208, Institut Camille Jordan, 43 blvd. du 11 novembre 1918, F--69622 Villeurbanne cedex, France}
\email{brehier@math.univ-lyon1.fr}

\author{David Cohen}
\address{Department of Mathematical Sciences, Chalmers University of Technology and University of
Gothenburg, SE--41296 Gothenburg, Sweden}
\email{david.cohen@chalmers.se}

\begin{document}

\begin{abstract}
We analyse a splitting integrator for the time discretization of the Schr\"odinger 
equation with nonlocal interaction cubic nonlinearity 
and white noise dispersion. We prove that this time integrator has order of convergence one in the $p$-th mean sense, 
for any $p\geq1$ in some Sobolev spaces.   
We prove that the splitting schemes preserves the $L^2$-norm, which is a crucial property for the proof of the 
strong convergence result. 
Finally, numerical experiments 
illustrate the performance of the proposed numerical scheme. 
\end{abstract}

\maketitle
{\small\noindent 
{\bf AMS Classification.} 65C30. 65J08. 60H15. 60H35. 60-08. 35Q55

\bigskip\noindent{\bf Keywords.} Stochastic partial differential equations. 
Stochastic Schr\"odinger equations. White noise dispersion. 
Nonlocal interaction cubic nonlinearity. Locally Lipschitz nonlinearity. 
Geometric numerical integration. Splitting integrators. 
Strong convergence rates. 

\section{Introduction}\label{intro}

We consider the time discretization by a splitting scheme for the following class of nonlinear Schr\"odinger equations with white noise dispersion 
\begin{align}\label{prob}
\begin{split}
\ii\dd u(t)+\Delta u(t)\circ\dd \beta(t)+V[u(t)]u(t)\,\dd t&=0\\
u(0)&=u_0,
\end{split}
\end{align}
where the unknown $u=u(t,\cdot)$, with $t\geq0$, 
is a complex valued random process defined on $\R^d$,  
$\Delta u=\displaystyle\sum_{j=1}^d\frac{\partial^2 u}{\partial x_j^2}$ denotes the Laplacian in $\R^d$, 
and $\beta=\beta(t)$ is a real-valued standard Brownian motion. 
The nonlinearity $\Psi_0(u)=V[u]u$ in the 
Stochastic Partial Differential Equation (SPDE)~\eqref{prob}   
is a nonlocal interaction cubic nonlinearity, $V[u]=V\star|u|^2=\displaystyle\int V(\cdot-x)|u(x)|^2\,\dd x$, 
where $\star$ denotes the convolution operator and the real-valued mapping $V\colon\R^d\to\R$ is at 
least continuous and bounded, more precise regularity conditions are imposed below. 
Such long-range interaction is a smooth version of the nonlinearity in the 
(deterministic) Schr\"odinger--Poisson equation, see for instance \cite{MR2429878}. 
Splitting schemes for Schr\"odinger equations driven by additive space-time noise 
with this type of nonlinearities were recently studied in \cite{bc20}. 
Observe that, the case of power-law nonlinearities 
cannot be treated by the techniques employed in the present publication.
The SPDE~\eqref{prob} is understood in the Stratonovich sense, 
using the $\circ$ symbol for the Stratonovich product.

Theoretical results on well-posedness of the SPDE \eqref{prob} are relatively scarce 
and given mostly for the case of a power-law nonlinearity ($|u|^{2\sigma}u$ for $\sigma$ a positive real number) 
in place of the nonlocal interaction nonlinearities considered in this article. 
For instance, it has been shown that SPDEs of the type~\eqref{prob} with power-law nonlinearities have solutions in $H^1$ 
for dimension $d=1$ and $\sigma=2$, \cite[Theorem 2.2]{MR2652190}, and for $\sigma<2/d$ in any dimension, 
\cite[Theorem 2.3]{MR2832639}. 

To the best of our knowledge, no strong convergence rates are know for a time discretization 
of the SPDE \eqref{prob} with the considered type of locally Lipschitz nonlinearity. 
However, strong convergence results have been proved in  the 
case of a globally Lipschitz nonlinearity in place of the above nonlinearity. 
In addition, rates of convergence in probability for a pure cubic nonlinearity in place 
of the above nonlocal interaction cubic nonlinearity have also been obtained. 
We now review these known convergence results. The work \cite{m06} studies a Lie--Trotter splitting integrator.  
The mean-square order of convergence of this explicit numerical method is proven 
to be at least $1/2$ for a (truncated) Lipschitz nonlinearity \cite[Sect. 5 and 6]{m06}. 
Furthermore, \cite{m06} conjectures that this splitting 
scheme should have strong order one, and supports this conjecture numerically. 
Sharp order estimates for the same splitting scheme (but applied to a more general problem) 
were recently presented in the preprint \cite{marty:hal-02314915}.  
The authors of \cite{MR3312594} study a semi-implicit Crank--Nicolson scheme. 
In particular, they show that this time integrator has mean-square order of convergence 
one for a truncated problem and order of convergence in probability one in the case of a cubic nonlinearity. 
For the same problem, the same convergence rates are obtained for a multi-symplectic integrator in \cite{MR3649275} 
and for an explicit exponential scheme in \cite{MR3736655}. 
We conclude this list of references with the recent work \cite{hks20} 
which considers a randomized exponential integrator for time discretization of 
related (non-random) nonlinear modulated Schr\"odinger equations.

In the present publication, we consider an explicit splitting integrator for 
an efficient time discretization of the nonlinear stochastic Schr\"odinger 
equation \eqref{prob}. In essence, the main principle of splitting integrators is to decompose the vector field of the 
original differential equation in several parts, such that the arising subsystems are exactly (or easily) 
integrated. We refer interested readers to \cite{MR2840298,MR3642447,MR2009376} for details 
on splitting schemes for ordinary (partial) differential equations. 
The splitting scheme considered in this publication is given by 
\[
u_{n+1}=\e^{\ii(\beta(t_{n+1})-\beta(t_n))\Delta}\bigl(\e^{\ii\tau V[u_n]}u_n\bigr),
\]
where $\tau$ denotes the time-step size, $t_n=n\tau$, and $u_n\approx u(t_n)$, 
see Equation~\eqref{split} below for details.

The main result of this paper is a strong convergence result for the explicit and easy to implement splitting integrator for the time discretization of \eqref{prob} defined above, see Section~\ref{sec-num} for a precise statement. 
Note that the nonlocal interaction cubic nonlinearity is only locally Lipschitz continuous in the Sobolev spaces $H^m(\R^d,\C)$, where $m=0,1,2,\ldots$, and not globally Lipschitz as in the above references, see Section~\ref{sec-not} for details. 
Theorem~\ref{th-ms} states that the splitting scheme converges with order $1$ in the $L^p(\Omega,H^m(\R^d,\C))$ sense, 
for all $p\in[1,\infty)$. One also obtains convergence with order $1$ in probability and in the almost sure sense. 
A crucial property for showing these results is the fact that the splitting scheme exactly 
preserves the $L^2$-norm as does the exact solution to \eqref{prob}, see Proposition~\ref{prop-ex} 
and Proposition~\ref{prop-norm2}.
To the best of our knowledge, this is the first strong convergence result obtained 
for a time discretization scheme applied to the nonlinear Schr\"odinger equation 
with white noise dispersion with a non-globally Lipschitz continuous nonlinearity.

In order to show such convergence results, we begin the exposition by introducing some 
notations and recalling useful results in Section~\ref{sec-not}. 
Section~\ref{sec-ex} then provides various properties of the exact solution to the SPDE \eqref{prob}. 
After that, we present the splitting scheme and analyse its strong convergence in Section~\ref{sec-num}. 
The proof of the main convergence result is given in Section~\ref{sec-error}. 
Several numerical experiments illustrating the main properties of the proposed numerical scheme are presented 
in Section~\ref{sec-numexp}. 

Throughout this article, we denote by $C$ a generic constant 
that may vary from line to line. Furthermore, we set 
$\N=\{1,2,\ldots\}$ and $\N_0=\{0,1,\ldots\}$. Finally, the initial 
value $u_0$ of the SPDE \eqref{prob} is assumed to be non-random for ease of 
presentation. The results of this paper can be extended to the case of random 
$u_0$ (independent of the given Brownian motion and with appropriate moment bounds).

\section{Setting and useful results}\label{sec-not}

We denote the classical Lebesgue space of complex functions by
$L^2=L^2(\R^d,\C)$, endowed with its real vector space structure,
and with the inner product 
$$
(u,v)=\ReT\int_{\R^d}u\bar v\,\dd x=\ReT\int u\bar v\,\dd x
$$
as well as its norm denoted by $\norm{\cdot}_{L^2}$. 
For $m\in\N$, we further denote by $H^m=H^m(\R^d,\C)$ the Sobolev space of functions in $L^2$, with weak derivatives of order $1,\ldots,m$ in $L^2$. The Fourier transform of a 
tempered distribution $v$ is denoted by $\widehat v$. 
With this notation, $H^m$ is the 
Sobolev space of tempered distributions $v$ such that $(1+\abs{\zeta}^2)^{m/2}\widehat v\in L^2$. 
The Sobolev space $H^m$ is equipped with the norm defined by
\[
\norm{v}_{H^m}^2=\sum_{|\alpha|\le m}\norm{\partial_\alpha v}_{L^2}^2,
\]
where $\alpha=\left(\alpha_1,\ldots,\alpha_d\right)\in\N_0^d$ is a multi-index and $|\alpha|=\displaystyle\sum_{i=1}^{d}\alpha_i$. Note that, if $m_1\le m_2$, one has $H^{m_1}\subset H^{m_2}$ and $\norm{v}_{H^{m_1}}\le \norm{v}_{H^{m_2}}$ for all $v\in H^{m_2}$. If $\alpha$ and $\gamma$ are two multi-indices, it is said that $\gamma\le \alpha$ if $\gamma_i\le \alpha_i$ for all $i=1,\ldots,d$. If $\gamma\le\alpha$, 
we also introduce the notation $\alpha-\gamma=\displaystyle\left(\alpha_i-\gamma_i\right)_{1\le i\le d}$ and  
$\displaystyle\binom{\alpha}{\gamma}=\prod_{i=1}^{d}\binom{\alpha_i}{\gamma_i}$. 

The Banach space $\mathcal C^m=\mathcal C^m(\R^d,\C)$ of complex-valued functions of class $\mathcal{C}^m$ is equipped with the norm 
\[
\|v\|_{\mathcal{C}^m}=\underset{|\alpha|\le m}\sup~\underset{x\in\R^d}\sup~|\partial_\alpha v(x)|.
\]

Let us recall a version of the Leibniz rule. See for instance \cite[Theorem 1, Sect. 5.2.3]{MR2597943} for a proof in the case of smooth compactly supported functions $v\in\mathcal{C}_c^\infty$, and \cite[Theorem 8.25]{MR2987297} for the argument to extend the result to  $v\in\mathcal{C}^m$.

\begin{lemma}[Leibniz rule]\label{lem:Leibniz}
For all $m\in\N$, there exists $C_m\in(0,\infty)$ such that for all $u\in H^m$ and $v\in\mathcal{C}^m$, one has $uv\in H^m$, and
\[
\norm{uv}_{H^m}\le C_m\norm{u}_{H^m}\norm{v}_{\mathcal{C}^m}.
\]
In addition, Leibniz rule holds: for all $\alpha\in\N_0^d$ with $|\alpha|\le m$, one has
\[
\partial_\alpha(uv)=\sum_{\gamma\le \alpha}\binom{\alpha}{\gamma}\partial_\gamma u\partial_{\alpha-\gamma}v.
\]
\end{lemma}

Let $\beta=\bigl(\beta(t)\bigr)_{t\ge 0}$ be a standard real-valued Brownian Motion defined on a filtered probability space $(\Omega,\mathcal{F},\mathbb{P},\{\mathcal{F}_t\}_{t\geq 0})$ satisfying the usual conditions.

For all $t,s\ge 0$, define the operator $S(t,s)$ as follows:
\begin{equation}\label{defS}
S(t,s)=\e^{\ii(\beta(t)-\beta(s))\Delta}.
\end{equation}
In Fourier variables, one has the expression
$$
\reallywidehat{S(t,s)v}(\zeta)=\exp\left(-\ii\abs{\zeta}^2(\beta(t)-\beta(s))\right)\widehat v(\zeta),
$$
for all $t,s\ge 0$, all $\zeta\in\R^d$ and any $v\in H^m$, $m\in\N_0$. 

The operators $S(t,s)$ for $t\ge s$ play an important role in this work: if $v$ is a $\mathcal{F}_s$-measurable random function with values in $H^{m}$, then $t\mapsto v_s(t)=S(t,s)v$ is the solution of the stochastic linear Schr\"odinger equation
\[
\ii\dd v_s(t)+\Delta v_s(t)\circ \dd\beta(t)=0,\quad t\ge s,
\]
with $v_s(s)=v$.

Two properties of the operators $S(t,s)$ will be used repeatedly in this article. 
\noindent First, for all $m\in\N_0$, all $t,s\ge 0$ and all $v\in H^m$, one has the isometry property
\begin{equation}\label{preservS}
\norm{S(t,s)v}_{H^m}=\norm{v}_{H^m}.
\end{equation}
\noindent Second, for all $m\in\N_0$, all $t,s\ge 0$ and all $v\in H^{m+2}$, one has
\begin{equation}\label{arnaud0}
\norm{S(t,s)v-v}_{H^m}\le |\beta(t)-\beta(s)|\norm{v}_{H^{m+2}}.
\end{equation}

Let us now study properties of the nonlinearity in the SPDE \eqref{prob} defined by
\begin{equation*}
\Psi_0(u)=V[u]u=\bigl(V\ast|u|^2\bigr)u.
\end{equation*}

If $V\in\mathcal{C}^m$, the mapping $\Psi_0:H^m\to H^m$ is well-defined and is locally Lipschitz continuous. More precisely, one has the following result.

\begin{lemma}\label{lemmaNL}
Let $m\in\N_0$ and assume that $V\in\mathcal{C}^m$. There exists $C_m(V)\in(0,\infty)$ such that the following properties hold.

First, for all $u\in H^m$, $\Psi_0(u)\in H^m$ and one has 
\begin{equation}\label{eq:lemmaNL-bound}
\norm{\Psi_0(u)}_{H^m}\le C_m(V)\|u\|_{L^2}^2\|u\|_{H^m}.
\end{equation}
In addition, $\Psi_0$ is locally Lipschitz continuous in $H^m$: for all $u_1,u_2\in H^m$, one has
\begin{equation}\label{eq:lemmaNL-Lip}
\norm{\Psi_0(u_2)-\Psi_0(u_1)}_{H^m}\le C_m(V)\left(\norm{u_1}_{H^m}^2+\norm{u_2}_{H^m}^2\right)\norm{u_2-u_1}_{H^m}.
\end{equation}
Finally, $\Psi_0$ is twice differentiable, and its first and second order derivatives satisfy the following result: for all $u,h,k\in H^m$, one has
\begin{equation}\label{eq:lemmaNL-d1}
\norm{\Psi_0'(u).h}_{H^m}\le C_m(V)\norm{u}_{L^2}\norm{u}_{H^m}\norm{h}_{H^m}
\end{equation}
and
\begin{equation}\label{eq:lemmaNL-d2}
\norm{\Psi_0''(u).(h,k)}_{H^m}\le C_m(V)\norm{u}_{L^2}\norm{h}_{H^m}\norm{k}_{H^m}.
\end{equation}
\end{lemma}

\begin{proof}[Proof of Lemma~\ref{lemmaNL}]
Let $m\in\N_0$ be fixed.

Using the definition of $\Psi_0$, Leibniz rule (Lemma~\ref{lem:Leibniz}) 
and the property $\norm{V[u]}_{\mathcal{C}^m}=\norm{V\ast|u|^2}_{\mathcal{C}^m}\le \norm{V}_{\mathcal{C}^m}\norm{u}_{L^2}^2$, 
the proof of~\eqref{eq:lemmaNL-bound} is straightforward: for all $u\in H^m$, one has
\[
\norm{\Psi_0(u)}_{H^m}=\norm{V[u]u}_{H^m}\le C_m\norm{V[u]}_{\mathcal{C}^m}\norm{u}_{H^m}\le C_m\norm{V}_{\mathcal{C}^m}\norm{u}_{L^2}^2\norm{u}_{H^m}.
\]
To prove~\eqref{eq:lemmaNL-d1} and~\eqref{eq:lemmaNL-d2}, note that the expressions for the derivatives are given by
\begin{align*}
\Psi_0'(u).h&=V[u]h+2V\ast\left(\ReT(\bar uh)\right)u\\
\Psi_0''(u).(h,k)&=4V\ast\left(\ReT(\bar kh)\right)u + 2V\ast\left(\ReT(\bar uk)\right)h+
2V\ast\left(\ReT(\bar uh)\right)k.
\end{align*}
Using Leibniz rule (Lemma~\ref{lem:Leibniz}) again, one obtains
\begin{align*}
\norm{\Psi_0'(u).h}_{H^m}&\le \norm{V[u]}_{\mathcal{C}^m}\norm{h}_{H^m}+2\norm{V\ast\left(\ReT(\bar uh)\right)}_{\mathcal{C}^m}\norm{u}_{H^m}\\
&\le \norm{V}_{\mathcal{C}^m}\norm{u}_{L^2}^2\norm{h}_{H^m}+2\norm{V}_{\mathcal{C}^m}\norm{u}_{L^2}\norm{h}_{L^2}\norm{u}_{H^m},
\end{align*}
and
\begin{align*}
\norm{\Psi_0''(u).(h,k)}_{H^m}&\le 4\norm{V\ast\left(\ReT(\bar kh)\right)}_{\mathcal{C}^m}\norm{u}_{H^m}+2\norm{V\ast\left(\ReT(\bar uk)\right)}_{\mathcal{C}^m}\norm{h}_{H^m}+2\norm{V\ast\left(\ReT(\bar uh)\right)}_{\mathcal{C}^m}\norm{k}_{H^m}\\
&\le 4\norm{V}_{\mathcal{C}^m}\left( \|u\|_{H^m}\|h\|_{L^2}\|k\|_{L^2}+
\|u\|_{L^2}\|h\|_{H^m}\|k\|_{L^2}+\|u\|_{L^2}\|h\|_{L^2}\|k\|_{H^m} \right).
\end{align*}

Finally, in order to prove~\eqref{eq:lemmaNL-Lip}, it suffices to write
\[
\Psi_0(u_2)-\Psi_0(u_1)=\int_0^1 \Psi_0'\bigl((1-\xi)u_1+\xi u_2\bigr).(u_2-u_1)\,\dd \xi,
\]
and to use~\eqref{eq:lemmaNL-d1}. One then obtains 
\[
\norm{\Psi_0(u_2)-\Psi_0(u_1)}_{H^m}\le C_m(V)(\norm{u_1}_{H^m}^2+\norm{u_2}_{H^m}^2)\norm{u_2-u_1}_{H^m}.
\]
This concludes the proof of Lemma~\ref{lemmaNL}.

\end{proof}

\section{Properties of the exact solution}\label{sec-ex}
In this section, we provide a well-posedness result and some properties of the exact solution $u(t)$ of the nonlinear Schr\"odinger equation with white noise dispersion \eqref{prob}. 

\begin{proposition}\label{prop-ex}
Assume that $V\in\mathcal{C}^0$.

For any (non-random) initial condition $u_0\in L^2$, 
there exists a unique mild solution $(u(t))_{t\ge0}$ of the 
Schr\"odinger with white noise dispersion \eqref{prob} in $L^2$, which means that for all $t\ge 0$ one has
\begin{align}\label{mild}
u(t)=S(t,0)u_0+\ii\int_0^tS(t,r)\left(V[u(r)]u(r)\right)\,\dd r,
\end{align}
where $\bigl(S(t,s)\bigr)_{t\ge s\ge 0}$ is defined by~\eqref{defS}.

In addition, one has conservation of 
the $L^2$-norm: for all $t\ge 0$, one has almost surely
\begin{equation}\label{preservL2-ex}
\norm{u(t)}_{L^2}=\norm{u_0}_{L^2}.
\end{equation} 

Furthermore, the SPDE \eqref{prob} is a stochastic Hamiltonian system, in the sense of \cite[Section~2]{MR3649275}, 
and thus its solution preserves the 
stochastic symplectic structure 
$$
\bar\omega=\int_{\R^d}\dd p\wedge \dd q\,\dd x \quad\text{almost surely},
$$
where the overbar on $\omega$ is a reminder that the two-form $\dd p\wedge\dd q$ 
(with differentials made with respect to the initial value) is integrated over $\R^d$. 
Here, $p(t)=\ReT(u(t))$ and $q(t)=\ImT(u(t))$ denote the real and imaginary parts of $u(t)$. 

Moreover, one can bound the solution in $\H^m$ in the following sense. Let $m\in\N$ and assume that $V\in\mathcal{C}^m$. 
There exists $C_m(V)\in(0,\infty)$ such that if $u_0\in H^m$, then almost surely $u(t)\in H^m$ for all $t\ge 0$, 
and 
\begin{equation}\label{borneHm-ex}
\norm{u(t)}_{H^m}\le \e^{C_m(V)\norm{u_0}_{L^2}^2t}\norm{u_0}_{H^m}.
\end{equation}
Finally, for all $u_0\in H^{m+2}$, $T\in(0,\infty)$ and $p\in[1,\infty)$, there exists $C_p(T,\norm{u_0}_{H^{m+2}})\in(0,\infty)$ such that for all $0\leq t_1\leq t_2\leq T$, one has
\begin{equation}\label{tempregul}
\left(\E[\norm{u(t_2)-u(t_1)}_{H^m}^p]\right)^{\frac1p}\le C_p(T,\norm{u_0}_{H^{m+2}})(t_2-t_1)^{\frac12}.
\end{equation}
\end{proposition}

\begin{proof}
Since $\Psi_0$ is locally Lipschitz continuous from $H^m$ to $H^m$, 
if $V\in\mathcal{C}^m$, for all $m\in\N$, local well-posedness of mild solutions in $H^m$ is a standard result.

To prove that solutions to \eqref{prob} are global, we use a truncation argument. 
Let $\theta\colon[0,\infty)\to[0,1]$ be a compactly supported Lipschitz continuous function, 
such that $\theta(x)=1$ for $x\in[0,1]$. For any $R\in(0,\infty)$, set $V^R(u)=\theta(R^{-1}\norm{u}_{L^2})V[u]$ 
and $F^R(u)=V^R(u)u$. The mapping $F^R$ is globally Lipschitz continuous, and the SPDE 
\[
\ii\dd u^R(t)  + \Delta u^R(t)\circ\dd \beta(t) + F^R(u^R(t))\,\dd t=0,
\]
with initial condition $u^R(0)=u_0$, thus admits a unique global mild solution $\left(u^R(t)\right)_{t\in[0,T]}$, 
where $T$ is an arbitrary positive real number. Applying It\^o's formula 
to a regularization of $u^R(t)$ as in the proof of \cite[Theorem 4.1]{MR2652190} for instance,
one checks that $\norm{u^R(t)}_{L^2}=\norm{u^R(0)}_{L^2}$ for all $t\in[0,T]$. 
Choosing $R>\norm{u_0}_{L^2}$ shows 
that one can define $u(t)=u^R(t)$ for all $t\ge 0$. 
Then $u(t)$ is the unique solution on $[0,T]$ of the fixed point equation~\eqref{mild}, 
{\it i.\,e.} $u(t)$ is the unique mild solution of~\eqref{prob}, 
and one has the preservation of the $L^2$-norm~\eqref{preservL2-ex}.

The fact that the problem \eqref{prob} is a stochastic Hamiltonian system is seen, 
exactly as in \cite[Sect. 2]{MR3649275}, by considering its real and imaginary parts and 
observing that the obtained differential equations are indeed stochastic Hamiltonian systems. 
The preservation of the stochastic symplectic structure follows also as in \cite[Sect. 2]{MR3649275} since 
the potential $V$ in \eqref{prob} is real-valued 
(as opposed to a power-law nonlinearity $|u|^{2\sigma}$ in the above reference). 

Let us now prove the bound in $H^m$, see~\eqref{borneHm-ex}. Using Lemma~\ref{lemmaNL}, one obtains
\[
\norm{V[u]u}_{H^m}\le C\norm{V[u]}_{\mathcal{C}^m}\norm{u}_{H^m}\le C\norm{V}_{\mathcal{C}^m}\norm{u}_{L^2}^2\norm{u}_{H^m}.
\]
Then, using the isometry property for $S(t,s)$ in $H^m$ (see~\eqref{preservS}), the mild formulation~\eqref{mild} and 
the preservation of the $L^2$-norm~\eqref{preservL2-ex}, one then obtains
\begin{align*}
\norm{u(t)}_{H^m}&\le \norm{u_0}_{H^m}+\int_0^t \norm{V[u(s)]u(s)}_{H^m}\,\dd s\\
&\le \norm{u_0}_{H^m}+C\int_0^t \norm{u(s)}_{L^2}^2\norm{u(s)}_{H^m}\,\dd s\\
&\le \norm{u_0}_{H^m}+C\int_0^t \norm{u_0}_{L^2}^2 \norm{u(s)}_{H^m}\,\dd s.
\end{align*}
Applying Gronwall's lemma then yields~\eqref{borneHm-ex}.

It remains to establish the temporal regularity property~\eqref{tempregul}. Using the mild formulation~\eqref{mild}, and the isometry property~\eqref{preservS}, one obtains
\begin{align*}
\norm{u(t_2)-u(t_1)}_{H^m}&\le \norm{\left(S(t_2,t_1)-I\right)u(t_1)}_{H^m}+\int_{t_1}^{t_2}\norm{V[u(t)]u(t)}_{H^m}\,\dd t\\
&\le \norm{\left(S(t_2,t_1)-I\right)u(t_1)}_{H^m}+C\norm{u_0}_{L^2}^2
\int_{t_1}^{t_2}\norm{u(t)}_{H^m}\,\dd t,
\end{align*}
using the inequality $\norm{V[u(t)]u(t)}_{H^m}\le C\norm{u(t)}_{L^2}^2\norm{u(t)}_{H^m}\le C\norm{u_0}_{L^2}^2\norm{u(t)}_{H^m}$, owing to the preservation of the $L^2$-norm~\eqref{preservL2-ex}.

Finally, using~\eqref{arnaud0}, the fact that $\bigl(\E[|\beta(t_2)-\beta(t_1)|^p]\bigr)^{\frac1p}\le 
C_p|t_2-t_1|^{\frac{1}{2}}$, and the bound for the exact solution in the $H^m$ norm~\eqref{borneHm-ex}, one obtains, for all $0\leq t_1\leq t_2\leq T$,
\begin{align*}
(\E[\norm{u(t_2)-u(t_1)}_{H^m}^p])^{\frac1p}&\leq C_p(t_2-t_1)^{1/2}\e^{C_{m+2}(V)\norm{u_0}_{L^2}^2 t_1}\norm{u_0}_{H^{m+2}}+
C(t_2-t_1)\e^{C_m(V)\norm{u_0}_{L^2}^2 t_2}\norm{u_0}_{H^m}^3\\
&\le C_p(T,\norm{u_0}_{H^{m+2}})(t_2-t_1)^{1/2}.
\end{align*}
This yields~\eqref{tempregul}, and concludes the proof of Proposition~\ref{prop-ex}.

\end{proof}

\section{Numerical analysis of the splitting scheme}\label{sec-num}

In this section, we propose and study an efficient time integrator for the SPDE \eqref{prob}. 
We state and prove some properties of the numerical solution, in particular preservation of the $L^2$-norm 
(Proposition~\ref{prop-norm2}). 
Furthermore, we state the main strong convergence result (Theorem~\ref{th-ms}) of the paper, 
namely that the splitting scheme has convergence rate $1$. Finally, we deduce various auxiliary results from the 
main theorem.

\subsection{Presentation of the splitting scheme}
Let $T>0$ be a fixed time horizon and an integer $N\geq1$. 
We define the step size of the numerical method by $\tau=T/N$  
and denote the discrete times by $t_n=n\tau$, for $n=0,\ldots,N$. 
Without loss of generality, we assume that $\tau\in(0,1)$.

The main idea of a splitting integrator for the SPDE \eqref{prob} is based on the observation 
that the vector field of the original problem can be decomposed in two parts 
(linear and nonlinear parts respectively) that are exactly integrated.

On the one hand, the solution of the linear stochastic evolution equation
\[
\ii\dd u(t)+\Delta u(t)\circ \dd \beta(t)=0,\quad u(0)=u_0
\]
is given by $u(t)=S(t,0)u_0$, where the random propagator $S(t,0)$ is defined by~\eqref{defS}.

On the other hand, the solution of the nonlinear evolution equation
\[
\ii\dd u(t)+V[u(t)]u(t)\,\dd t=0,\quad u(0)=u_0
\]
is given by $u(t)=\Phi_t(u_0)$, where for all $t\ge 0$ and all $u\in L^2$, one has
\begin{equation}\label{defPhi}
\Phi_t(u)=\e^{\ii tV[u]}u.
\end{equation}

The Lie-Trotter splitting strategy yields the definition of the 
following time integrator for the nonlinear Schr\"odinger equation with white noise dispersion~\eqref{prob}:
\begin{equation}\label{split}
u_{n+1}=S(t_{n+1},t_n)\Phi_\tau(u_n).
\end{equation}

The following notation will be used in the sequel: for all $\tau\in(0,1)$ and $u\in L^2$, set
\begin{equation*}
\Psi_\tau(u)=\frac{\Phi_\tau(u)-u}{\ii\tau}.  
\end{equation*}

\subsection{Properties of the numerical solution}

This subsection lists useful properties of the numerical solution given by the splitting scheme \eqref{split}. 

\textbf{Conservation of the $\mathbf L^{\mathbf2}$-norm}. 
The splitting scheme exactly preserves the $L^2$-norm as does the exact solution to the SPDE \eqref{prob}, 
see equation \eqref{preservL2-ex} in Proposition~\ref{prop-ex}.  
This conservation property plays a crucial role in the error analysis presented below.
\begin{proposition}\label{prop-norm2}
Let $u_0\in L^2$, $\tau\in(0,1)$ and let $\bigl(u_n\bigr)_{n\in\N_0}$ be given by the splitting scheme~\eqref{split}, 
one then has conservation of the $L^2$-norm: for all $n\in\N$, one has almost surely
\begin{equation}\label{preservL2-split}
\norm{u_n}_{L^2}=\norm{u_0}_{L^2}.
\end{equation}
\end{proposition}

\begin{proof}
Using the isometry property~\eqref{preservS}, then the definition~\eqref{defPhi} of the flow $\Phi_\tau$, 
a direct computation from the definition of the scheme~\eqref{split} gives for all $n\in\N_0$
\[
\norm{u_{n+1}}_{L^2}=\norm{S(t_{n+1},t_n)\Phi_\tau(u_n)}_{L^2}=\norm{\Phi_\tau(u_n)}_{L^2}=\norm{u_n}_{L^2}.
\]
A straightforward recursion argument concludes the proof.
\end{proof}

\textbf{Bounds for the numerical solution in $\mathbf H^{\mathbf m}$}. 
Proposition~\ref{propbnd} below states almost sure upper bounds for the numerical solution 
$\norm{u_n}_{H^m}$, for all $n\in\N_0$ and $m\in\N$.

\begin{proposition}\label{propbnd}
Let $m\in\N$ and assume that $V\in\mathcal{C}^m$. There exists $C_m(V)\in(0,\infty)$, such that for any initial condition $u_0\in H^m$, the numerical solution $u_n$ defined by the splitting scheme~\eqref{split} satisfies the following upper bound: for all $n\in\N_0$, 
one has almost surely
\begin{equation}\label{bnd}
\norm{u_n}_{H^m}\leq \e^{C_m(V)t_n\norm{u_0}_{L^2}^{2m}}\norm{u_0}_{H^m}.
\end{equation}
\end{proposition}

The proof of Proposition~\ref{propbnd} requires the following auxiliary result.
\begin{lemma}\label{lemma-bound}
Let $m\in\N$ and assume that $V\in\mathcal{C}^m$.

There exists $C_m(V)\in(0,\infty)$ such that for all $\tau\in(0,1)$ and all $u\in H^m$, one has
\[
\norm{\Phi_\tau(u)}_{H^m}\le \left(1+C_m(V)\tau(1+\norm{u}_{L^2}^{2m})\right)\norm{u}_{H^m}.
\]
\end{lemma}

\begin{proof}[Proof of Lemma~\ref{lemma-bound}]
Using the definition~\eqref{defPhi}, one has the identity $\Phi_\tau(u)=\e^{\ii\tau V[u]}u$, with $V[u]=V\ast|u|^2\in\mathcal{C}^m$ and $\e^{\ii\tau V[u]}\in \mathcal{C}^m$, since $V\in\mathcal{C}^m$.

The following expression holds: for all $u\in H^m$ one has
\[
\Phi_\tau(u)-u=\theta_\tau(V[u])u,
\]
where $\theta_\tau(y)=\e^{\ii \tau y}-1$ for all $y\in \R$. Applying the inequality from Lemma~\ref{lem:Leibniz}, one has, for all $u\in H^m$,
\[
\norm{\Phi_\tau(u)-u}_{H^m}\le C_m\norm{\theta_\tau(V[u])}_{\mathcal{C}^m}\norm{u}_{H^m}.
\]
It remains to study the behavior of $\norm{\theta_\tau(V[u])}_{\mathcal{C}^m}$. The auxiliary function $\theta_\tau$ satisfies the following properties: for all $y\in\R$, all $\tau\in(0,1)$ and all $k\in\N$,
\begin{align*}
|\theta_\tau(y)|&\le \tau |y|\\
|\theta_\tau^{(k)}(y)|&\le \tau^{k}.
\end{align*}
Using the Fa\`a di Bruno formula, one obtains the bounds
\[
\norm{\partial_\gamma \theta_\tau(V[u])}_{\mathcal{C}^0}\le C\begin{cases} \tau\norm{V[u]}_{\mathcal{C}^0}~,\quad \gamma=0,\\ 
C_{|\gamma|}\tau\left(1+\norm{V[u]}_{\mathcal{C}^{|\gamma|}}^{|\gamma|}\right)~,\quad \gamma\neq 0.\end{cases}
\]
Finally, using the inequality
\[\norm{V[u]}_{\mathcal{C}^{|\gamma|}}^{|\gamma|}=\norm{\partial_\gamma \left(V \ast|u|^2\right)}_{\mathcal{C}^{0}}^{|\gamma|}\le C_m(V)\norm{u}_{L^2}^{2m}
\]
if $1\le |\gamma|\le m$ then concludes the proof of Lemma~\ref{lemma-bound}.

\end{proof}

We are now in position to provide the proof of Proposition~\ref{propbnd}.
\begin{proof}[Proof of Proposition~\ref{propbnd}]
Using the definition~\eqref{split} of the splitting scheme, 
the isometry property~\eqref{preservS} of the random propagator $S(t_{n+1},t_n)$, and Lemma~\ref{lemma-bound}, one gets
\begin{align*}
\norm{u_{n+1}}_{H^m}&= \norm{S(t_{n+1},t_n)\Phi_\tau(u_n)}_{H^m}=\norm{\Phi_\tau(u_n)}_{H^m}\leq \left(1+C_m(V)
\norm{u_n}_{L^2}^{2m}\tau\right)\norm{u_n}_{H^m}.
\end{align*}
Using the preservation property~\eqref{preservL2-split} of the $L^2$-norm by the splitting integrator, 
see Proposition~\ref{prop-norm2}, one then obtains the following estimate: for all $n\in\N_0$
\[
\norm{u_{n+1}}_{H^m}\le \left(1+C_m(V)\norm{u_0}_{L^2}^{2m}\tau\right)\norm{u_n}_{H^m}.
\]
Finally, a straightforward recursion argument yields the following bound: for all $n\in\N_0$, one has
\[
\norm{u_n}_{H^m}\le \e^{C_m(V)t_n\norm{u_0}_{L^2}^{2m}}\norm{u_0}_{H^m}.
\]

All the estimates above hold in an almost sure sense. This concludes the proof of Proposition~\ref{propbnd}.

\end{proof}

\textbf{Numerical preservation of the stochastic symplectic structure}. 
As seen in Proposition~\ref{prop-ex}, the exact solution to the SPDE \eqref{prob} 
preserves a stochastic symplectic structure. 
The next result states that the same geometric structure is also preserved by the splitting scheme \eqref{split}. 
\begin{proposition}
Consider the numerical discretization of the Schr\"odinger equation with white noise dispersion \eqref{prob} 
by the splitting scheme \eqref{split}. Then, the splitting scheme preserves the stochastic symplectic structure
$$
\bar\omega^{n+1}=\bar\omega^{n}\quad\text{for}\quad n=0,\ldots,N-1\quad\text{almost surely},
$$
where $\bar\omega^n=\displaystyle\int_{\R^d}\dd p^n\wedge \dd q^n\,\dd x$ and $p^n$, resp. $q^n$, 
are the real, resp. imaginary parts of $u^n$.
\end{proposition}

\begin{proof}
The splitting integrator \eqref{split} is obtained by solving 
exactly sequentially the following differential equations:
$$
\ii\dd u(t)+V[u(t)]u(t)\,\dd t=0
$$
and 
$$
\ii\dd u(t)+\Delta u(t)\circ\dd \beta(t)=0.
$$
Considering the real and imaginary parts of these differential equations and using the fact that $V$ is real-valued, 
one gets
$$
\dd p(t)=-V[(p(t),q(t))]q(t)\,\dd t,\quad\dd q(t)=V[(p(t),q(t))]p(t)\,\dd t
$$
and
$$
\dd p(t)=-\Delta q(t)\circ \dd\beta(t),\quad \dd q(t)=\Delta p(t)\circ \dd\beta(t).
$$
The above problems are infinite-dimensional stochastic Hamiltonian systems in the sense of \cite[Eq. (6)]{MR3649275}. 
It thus follows, as in \cite[Prop. 3.3]{MR3649275}, that the splitting scheme preserves the stochastic symplectic structures 
of each of these Hamiltonian systems, as it is obtained as composition of symplectic maps, and hence the statement. 
\end{proof}

\subsection{Convergence results}

We are now in position to state the main result of this article.

\begin{theorem}\label{th-ms}
Let $\bigl(u(t)\bigr)_{t\ge 0}$, resp.~$\bigl(u_n\bigr)_{n\in\N_0}$, be the solutions of the stochastic Schr\"odinger equation~\eqref{prob}, resp. of the splitting scheme~\eqref{split}, with (non-random) initial condition $u_0$.

Let $m\in\N_0$ and assume that $V\in\mathcal{C}^{m+4}$. For all $p\in[1,\infty)$, all $T\in(0,\infty)$ and all $u_0\in H^{m+4}$, there exists $C_{m,p}(T,\norm{u_0}_{H^{m+4}})\in(0,\infty)$ such that, for all $\tau\in(0,1)$, one has
\begin{equation}\label{error}
\underset{0\le n\le N}\sup~\left(\E\left[\norm{u_n-u(t_n)}_{H^m}^{p}\right]\right)^{\frac1p}\leq C_{m,p}(T,\norm{u_0}_{H^{m+4}})\tau.
\end{equation}

\end{theorem}

The proof of Theorem~\ref{th-ms} is postponed to Section~\ref{sec-error}.

Note that contrary to previous works in the literature, \cite{m06,MR3312594,MR3736655,MR3649275}, 
concerning the analysis of 
numerical schemes for stochastic Schr\"odinger equations 
with white noise dispersion with a globally Lipschitz continuous nonlinearity, in Theorem~\ref{th-ms} 
we consider the moments of arbitrary order $p\in[1,\infty)$, instead of only $p=2$ (mean-square error). 
We also consider the error in the $H^m$ norm, for arbitrary $m\in\N_0$. 
We could use the same strategy of proof as in those references when $p=2$, 
however we need to use a different strategy when $p\neq 2$ and directly consider the general case $p\in[1,\infty)$.

\begin{remark}
If the initial condition $u_0$ and the potential $V$ are less regular than in Theorem~\ref{th-ms}, it is possible to obtain the following result: assume that $u_0\in H^{m+2}$ and that $V\in\mathcal{C}^{m+2}$, then
\[
\underset{0\le n\le N}\sup~\left(\E\left[\norm{u_n-u(t_n)}_{H^m}^{p}\right]\right)^{\frac1p}\leq C_{m,p}(T,\norm{u_0}_{H^{m+2}})\tau^{\frac12}.
\]
\end{remark}

As immediate consequences of the main result of this article we obtain the following corollaries. 
\begin{corollary}
Under the assumptions of Theorem~\ref{th-ms}, one obtains the following error estimate: for all $\varepsilon\in(0,1)$, there exists $C_{m,p,\varepsilon}(T,\norm{u_0}_{H^{m+4}})\in(0,\infty)$ such that, for all $\tau\in(0,1)$, one has
\begin{equation}\label{error-sup}
\left(\E\left[\underset{0\le n\le N}\sup~\norm{u_n-u(t_n)}_{H^m}^{p}\right]\right)^{\frac1p}\leq C_{m,p,\varepsilon}(T,\norm{u_0}_{H^{m+4}})\tau^{1-\varepsilon}.
\end{equation}
\end{corollary}

\begin{proof}
The second error estimate~\eqref{error-sup} follows from the first error estimate~\eqref{error} by an elementary argument. 
Let $\varepsilon\in(0,1)$ and $p\in[1,\infty)$, and choose $q>\max(p,\varepsilon^{-1})$, then using~\eqref{error} one obtains
\begin{align*}
\E\left[\underset{0\le n\le N}\sup~\norm{u_n-u(t_n)}_{H^m}^{q}\right]&\le \sum_{n=0}^{N}\E\left[\norm{u_n-u(t_n)}_{H^m}^{q}\right]\\
&\le \frac{T}{\tau}\bigl(C_{m,q}(T,\norm{u_0}_{H^{m+4}})\tau)^{q}\\
&\le TC_{m,q}(T,\norm{u_0}_{H^{m+4}})^q\tau^{q(1-\frac{1}{q})}\\
&\le C_{m,q}(T,\norm{u_0}_{H^{m+4}})\tau^{q(1-\varepsilon)},
\end{align*}
where we recall that $\tau=T/N$. 
Finally one obtains~\eqref{error-sup} as follows:
\[
\left(\E\left[\underset{0\le n\le N}\sup~\norm{u_n-u(t_n)}_{H^m}^{p}\right]\right)^{\frac1p}\le \left(\E\left[\underset{0\le n\le N}\sup~\norm{u_n-u(t_n)}_{H^m}^{q}\right]\right)^{\frac1q}\le C_{m,p,\varepsilon}(T,\norm{u_0}_{H^{m+4}})\tau^{1-\varepsilon}.
\]
\end{proof}

The argument described above gives a slight reduction in the order of convergence from $1$ to $1-\varepsilon$, with arbitrarily small $\varepsilon>0$. It may be possible to obtain~\eqref{error-sup} with $\varepsilon=0$ using refined arguments in the analysis of the error. To keep the presentation simple, this is not performed in the sequel.

The fact that the first error estimate~\eqref{error} holds with arbitrarily large $p$ is important and allows us to choose arbitrarily small $\varepsilon$. If one applies the argument detailed above only when $p=2$ for instance, one obtains an order of convergence $\frac12$ in~\eqref{error-sup}.

\begin{corollary}\label{cor:pas}
Consider the stochastic Schr\"odinger equation \eqref{prob} on the time interval $[0,T]$ 
with solution denoted by $\left(u(t)\right)_{t\in[0,T]}$. 
Let $u_n$ be the numerical solution given by the splitting scheme \eqref{split} with time-step size $\tau$. 
Under the assumptions of Theorem~\ref{th-ms}, one has convergence in probability of order one
\[
\underset{C\to\infty}\lim~\underset{\tau\in(0,1)}\sup~\mathbb{P}\left(\norm{u_N-u(T)}_{H^m}\ge C\tau\right)=0,
\]
where we recall that $T=N\tau$. 

Moreover, consider the sequence of time-step sizes given by $\tau_L=\frac{T}{2^L}$, $L\in\N$. 
Then, for every $\varepsilon\in(0,1)$, there exists an almost surely finite random variable 
$C_\varepsilon$, such that for all $L\in\N$ one has
\[
\norm{u_{2^L}-u(T)}_{H^m}\le C_{\varepsilon}\left(\frac{T}{2^L}\right)^{1-\varepsilon}.
\]
\end{corollary}

\begin{proof}
The result on convergence in probability is a straightforward consequence of 
Markov's inequality followed by Theorem~\ref{th-ms}:
\begin{align*}
\mathbb{P}\left(\norm{u_N-u(T)}_{H^m}\ge C\tau\right)&\le 
\frac{\E\left[\norm{u_N-u(T)}_{H^m}\right]}{C\tau}=\frac{C_{m,1}(T,\norm{u_0}_{H^{m+4}})}{C}\underset{C\to\infty}\to 0.
\end{align*}
To get the result on almost sure convergence, it suffices to observe that (again by applying Theorem~\ref{th-ms})
\[
\sum_{\ell=0}^{\infty}\frac{\E\left[\norm{u_{2^{\ell}}-u(T)}_{H^m}\right]}{\tau_{\ell}^{1-\varepsilon}}<\infty,
\]
thus $\frac{\norm{u_{2^L}-u(T)}_{H^m}}{\tau_L^{1-\varepsilon}}\underset{L\to\infty}\to 0$ almost surely. 
\end{proof}

\section{Error analysis: proof of Theorem~\ref{th-ms}}\label{sec-error}

Before proceeding with the proof of the error estimates~\eqref{error},  
let us state and prove an auxiliary result on the mappings $\Psi_0(u)=V[u]u$ and 
$\Psi_\tau(u)=\frac{\Phi_\tau(u)-u}{\ii\tau}$.

\begin{lemma}\label{techlemma}
Let $m\in\N_0$ and assume that $V\in\mathcal{C}^m$.

There exists $C_m(V)\in(0,\infty)$ such that for all $u\in H^m$ and all $\tau\in(0,1)$,
one has
\[
\norm{\Psi_\tau(u)-\Psi_0(u)}_{H^m}\le C_m(V)\tau\left(1+\norm{u}_{L^2}^{\max(4,2m)}\right)\norm{u}_{H^m}.
\]
\end{lemma}

\begin{proof}
Let us first observe that, by definitions of the operators $\Psi_\tau$ and $\Psi_0$, one has
\[
\Psi_\tau(u)-\Psi_0(u)=\Theta_\tau(V[u])u,
\]
where $\Theta_\tau(y)=\frac{\e^{\ii\tau y}-1-\ii\tau y}{\ii\tau}$ for all $y\in\R$.

Applying the inequality from Lemma~\ref{lem:Leibniz}, one obtains
\[
\norm{\Psi_\tau(u)-\Psi_0(u)}_{H^m}\le C_m\norm{\Theta_\tau(V[u])}_{\mathcal{C}^m}\norm{u}_{H^m}.
\]
It remains to study the behavior of $\norm{\Theta_\tau(V[u])}_{\mathcal{C}^m}$.

The auxiliary function $\Theta_\tau$ satisfies the following properties: for all $k\in\N_0$, there exists $C_k\in(0,\infty)$ such that, for all $y\in\R$, one has
\begin{itemize}
\item $|\Theta_\tau(y)|\le C_0\tau|y|^2$
\item $|\Theta_\tau'(y)|\le C_1\tau|y|$
\item $|\Theta_\tau^{(k)}(y)|\le C_k\tau^{k-1}\le C_k\tau$ for all integers $k\geq2$ and all $\tau\in(0,1)$.
\end{itemize}
Using the Fa\`a di Bruno formula, one obtains the bounds
\[
\norm{\partial_\gamma \Theta_\tau(V[u])}_{\mathcal{C}^0}\le C\begin{cases} \tau\norm{V[u]}_{\mathcal{C}^0}^2~,\quad \gamma=0,\\ 
\tau\left(1+\norm{V[u]}_{\mathcal{C}^{|\gamma|}}^{|\gamma|}\right)~,\quad \gamma\neq 0.\end{cases}
\]
Using the inequality
\[\norm{V[u]}_{\mathcal{C}^{|\gamma|}}^{|\gamma|}=\norm{\partial_\gamma V \ast|u|^2}_{\mathcal{C}^{0}}^{|\gamma|}\le C_m(V)\norm{u}_{L^2}^{2m}
\]
if $|\gamma|\le m$ then concludes the proof of Lemma~\ref{techlemma}.

\end{proof}

We are now in position to give the proof of Theorem~\ref{th-ms}.
\begin{proof}[Proof of Theorem~\ref{th-ms}]
Let us first perform a change of unknowns: for all $t\ge 0$ and $n\in\N_0$, set 
\[
v(t)=S(0,t)u(t)=S(t,0)^{-1}u(t)~,\quad\text{and}\quad v_n=S(0,t_n)u_n=S(t_n,0)^{-1}u_n,
\]
where $u(t)$ is the solution of~\eqref{prob} whereas $u_n$ is defined by the splitting scheme~\eqref{split}. Owing to the isometry property~\eqref{preservS} for the random propagator, one has the equality
\[
\norm{u_n-u(t_n)}_{H^m}=\norm{S(t_n,0)\left(v_n-v(t_n)\right)}_{H^m}=\norm{v_n-v(t_n)}_{H^m},
\]
for all $m\in \N_0$ and for all $n\in\N_0$. Thus, it is sufficient to prove estimates for the error $E_n=v_n-v(t_n)$.

Using the mild form~\eqref{mild} for $u(t)$ and the definition of the splitting scheme~\eqref{split}, for all $t\ge 0$ and $n\in\N_0$, one has the following expressions:
\begin{align*}
v(t)&=u_0+\ii\int_0^t S(0,s)\Psi_0(u(s))\,\dd s\\
v(t_{n+1})&=v(t_n)+\ii\int_{t_n}^{t_{n+1}}S(0,t)\Psi_0(u(t))\,\dd t\\
v_{n+1}&=S(0,t_n)\Phi_\tau(u_n)=v_n+\ii\tau S(0,t_n)\Psi_\tau(u_n).
\end{align*}
The expressions above then give the following decomposition of the error:
\[
E_{n+1}=v_{n+1}-v(t_{n+1})=E_n+\epsilon_n^1+\epsilon_n^2+\epsilon_n^3+\epsilon_n^4+\epsilon_n^5,
\]
with local error terms defined by
\begin{align*}
\epsilon_n^1&=\ii\tau S(0,t_n)\left(\Psi_\tau(u_n)-\Psi_0(u_n)\right)\\
\epsilon_n^2&=\ii\tau S(0,t_n)\left(\Psi_0(u_n)-\Psi_0(u(t_n))\right)\\
\epsilon_n^3&=\ii\int_{t_n}^{t_{n+1}}\left(S(0,t)-S(0,t_n)\right)\left(\Psi_0(u(t_n))-\Psi_0(u(t))\right)\,\dd t\\
\epsilon_n^4&=\ii\int_{t_n}^{t_{n+1}}\left(S(0,t_n)-S(0,t)\right)\Psi_0(u(t_n))\,\dd t\\
\epsilon_n^5&=\ii\int_{t_n}^{t_{n+1}}S(0,t_n)\left(\Psi_0(u(t_n))-\Psi_0(u(t))\right)\,\dd t.
\end{align*}
For $j=1,\ldots,5$, set $\displaystyle 
E_n^j=\sum_{k=0}^{n-1}\epsilon_k^j$. Then a straightforward recursion argument yields the equality
\[
E_n=\sum_{j=1}^{5}E_n^j=\sum_{j=1}^{5}\sum_{k=0}^{n-1}\epsilon_k^j.
\]
and applying Minkowski's inequality one obtains, for $p\in[1,\infty)$,
\[
\left(\E\left[\norm{E_n}_{H^m}^p\right]\right)^{\frac1p}\le \sum_{j=1}^{5}\left(\E[\norm{E_n^j}_{H^m}^p]\right)^{\frac1p}.
\]
It remains to prove error estimates for $\left(\E[\norm{E_n^j}_{H^m}^p]\right)^{\frac1p}$, for $j=1,\ldots,5$ and $p\in[2,\infty)$ 
(the case $p\in[1,2)$ is treated using H\"older's inequality). 
The estimates of the error terms for $j=1,2,3$ follow from straightforward arguments, whereas more work is required to deal with the cases $j=4$ and $j=5$ (in order to obtain order of convergence equal to $1$, 
instead of the order $1/2$ corresponding to the temporal H\"older regularity of 
the solution, see Equation~\eqref{tempregul}).

We now provide detailed error estimates of those five terms.

\noindent$\bullet$ Let us start with the first term. Using Minkowski's inequality and the isometry property~\eqref{preservS} of the random propagator, one has
\begin{align*}
\left(\E\left[\norm{E_n^1}_{H^m}^p\right]\right)^{\frac1p}&\le 
\sum_{k=0}^{n-1}\left(\E\left[\norm{\epsilon_k^1}_{H^m}^p\right]\right)^{\frac1p}
\le \sum_{k=0}^{n-1}\left(\E\left[\norm{\ii\tau S(0,t_k)\left(\Psi_\tau(u_k)-\Psi_0(u_k)\right) }_{H^m}^p\right]\right)^{\frac1p}
\\
&\le \tau\sum_{k=0}^{n-1}\left(\E\left[\norm{\Psi_\tau(u_k)-\Psi_0(u_k)}_{H^m}^p\right]\right)^{\frac1p}\\
\end{align*}
Applying Lemma~\ref{techlemma} and using first the preservation of the $L^2$-norm property~\eqref{preservL2-split} for the numerical scheme (Proposition~\ref{prop-norm2}), second the almost sure bound~\eqref{bnd} for the $H^m$-norm of the numerical solution (Proposition~\ref{propbnd}), one finally obtains, for all $n=0,\ldots,N$,
\begin{align*}
\left(\E\left[\norm{E_n^1}_{H^m}^p\right]\right)^{\frac1p}&\le C\tau^2\sum_{k=0}^{n-1}\left(\E\left[\left(1+\norm{u_k}_{L^2}^{\max(4,2m)}\right)^{p}\norm{u_k}_{H^m}^p\right]\right)^{\frac1p}\\
&\le C\tau^2\sum_{k=0}^{n-1}\left(1+\norm{u_0}_{L^2}^{\max(4,2m)}\right)^{p}\left(\E\left[\norm{u_k}_{H^m}^p\right]\right)^{\frac1p}\\
&\le C_{m,p}(T,\norm{u_0}_{H^m})\tau.
\end{align*}

\noindent$\bullet$ For the second term, using Minkowski's inequality and the isometry property~\eqref{preservS} of the random propagator, one has
\begin{align*}
\left(\E\left[\norm{E_n^2}_{H^m}^p\right]\right)^{\frac1p}&\le 
\sum_{k=0}^{n-1}\left(\E\left[\norm{\epsilon_k^2}_{H^m}^p\right]\right)^{\frac1p}
\le \sum_{k=0}^{n-1}\left(\E\left[\norm{\ii\tau S(0,t_k)\left(\Psi_0(u_k)-\Psi_0(u(t_k))\right)}_{H^m}^p\right]\right)^{\frac1p}
\\
&\le \tau\sum_{k=0}^{n-1}\left(\E\left[\norm{\Psi_0(u_k)-\Psi_0(u(t_k))}_{H^m}^p\right]\right)^{\frac1p}.
\end{align*}
Using the local Lipschitz continuity property~\eqref{eq:lemmaNL-Lip} of $\Psi_0$ (Lemma~\ref{lemmaNL}), then the almost sure bounds for the $H^m$ norm of the exact solution (Equation~\eqref{borneHm-ex} from Proposition~\ref{prop-ex}) and of the numerical solution (Equation~\eqref{bnd} from Proposition~\ref{propbnd}), one obtains, for all $n=0,\ldots,N$,
\begin{align*}
\left(\E\left[\norm{E_n^2}_{H^m}^p\right]\right)^{\frac1p}&\le C\tau\sum_{k=0}^{n-1} \left(\E\left[\left(\norm{u_k}_{H^m}^2+\norm{u(t_k)}_{H^m}^2\right)^p
\norm{u_k-u(t_k)}_{H^m}^p\right]\right)^{\frac1p}\\
&\le C(T,\norm{u_0}_{H^m})\tau\sum_{k=0}^{n-1}\left(\E\left[\norm{E_k}_{H^m}^p\right]\right)^{\frac1p}.
\end{align*}

\noindent$\bullet$ In order to estimate the third term, applying Minkowski's inequality yields 
\begin{align*}
\left(\E\left[\norm{E_n^3}_{H^m}^p\right]\right)^{\frac1p}&\le \sum_{k=0}^{n-1}\left(\E\left[\norm{\epsilon_k^3}_{H^m}^p\right]\right)^{\frac1p}.
\end{align*}
Using the inequality~\eqref{arnaud0} (combined with Cauchy--Schwarz's inequality) and the local Lipschitz continuity property~\eqref{eq:lemmaNL-Lip} of $\Psi_0$ (Lemma~\ref{lemmaNL}) then yields
\begin{align*}
\left(\E\left[\norm{\epsilon_k^3}_{H^m}^p\right]\right)^{\frac1p}
&\le \int_{t_k}^{t_{k+1}}\left(\E\left[ \norm{\left(S(0,t_k)-S(0,t)\right)
\left(\Psi_0(u(t))-\Psi_0(u(t_k))\right)}^p_{H^m} \right]\right)^{\frac1p}\,\dd t\\
&\le C\tau^{\frac12}\int_{t_k}^{t_{k+1}}\left(\E\left[\norm{\Psi_0(u(t))-\Psi_0(u(t_k))}_{H^{m+2}}^{2p}\right]\right)^{\frac{1}{2p}}\,\dd t\\
&\le C\tau^{\frac12}\int_{t_k}^{t_{k+1}}\left(\E\left[\left(\norm{u(t)}_{H^{m+2}}^2+\norm{u(t_k)}_{H^{m+2}}^2\right)^{2p}\norm{u(t)-u(t_k)}_{H^{m+2}}^{2p}\right]\right)^{\frac{1}{2p}}\,\dd t.
\end{align*}
Using the almost sure bound~\eqref{borneHm-ex} for the $H^{m+2}$-norm of the exact solution (Proposition~\ref{prop-ex}) and the temporal regularity estimate~\eqref{tempregul}, one finally obtains
\[
\left(\E\left[\norm{\epsilon_k^3}_{H^m}^p\right]\right)^{\frac1p}\le C(T,\norm{u_0}_{H^{m+4}})\tau^{\frac12}\int_{t_k}^{t_{k+1}}(t-t_k)^{\frac12}\,\dd t \le C(T,\norm{u_0}_{H^{m+4}})\tau^2,
\]
and finally one has, for all $n=0,\ldots,N$,
\[
\left(\E\left[\norm{E_n^3}_{H^m}^p\right]\right)^{\frac1p}\le 
\sum_{k=0}^{n-1}\left(\E\left[\norm{\epsilon_k^3}_{H^m}^p\right]\right)^{\frac1p}
\le C_{m,p}(T,\norm{u_0}_{H^{m+4}})\tau.
\]

\noindent$\bullet$ Let us now focus on the the fourth term. As explained above, one needs to be careful to obtain an order of convergence equal to $1$. Indeed, for the fourth term, applying~\eqref{arnaud0} directly (and appropriate bounds) would only give order of convergence $1/2$ of the splitting scheme.

Let us define auxiliary processes: for all $n\in\N_0$ and all $t\in[t_n,t_{n+1}]$, set 
\[
v_n(t)=S(0,t)\Psi_0(u(t_n)).
\] 
Note that $v_n(t)=S(t_n,t)v_n(t_n)=\e^{-\ii\left(\beta(t)-\beta(t_n)\right)\Delta}v_n(t_n)$, for all $t\in[t_n,t_{n+1}]$. As a consequence, the process $\left(v_n(t)\right)_{t\in[t_n,t_{n+1}]}$ is the solution of the linear stochastic evolution equation
\[
\dd v_n(t)=-\ii\Delta v_n(t)\circ \dd\beta(t)=-\ii\Delta v_n(t)\,\dd\beta(t)-\frac12\Delta^2 v_n(t)\,\dd t,
\]
with $v_n(t_n)=S(0,t_n)\Psi_0(u(t_n))$.

The local error term $\epsilon_n^4$ is rewritten as follows in terms of the auxiliary process $v_n$:
\begin{align*}
\epsilon_n^4&=\ii\int_{t_n}^{t_{n+1}}\left(S(0,t_n)-S(0,t)\right)\Psi_0(u(t_n))\,\dd t\\
&=\ii\int_{t_n}^{t_{n+1}}\left(v_n(t_n)-v_n(t)\right)\,\dd t\\
&=\int_{t_n}^{t_{n+1}}\int_{t_n}^{t}\Delta v_n(s)\,\dd\beta(s)\,\dd t-\frac{\ii}{2}\int_{t_n}^{t_{n+1}}\int_{t_n}^{t}\Delta^2 v_n(s)\,\dd s\,\dd t\\
&=\epsilon_n^{4,1}+\epsilon_n^{4,2},
\end{align*}
with
\begin{align*}
\epsilon_n^{4,1}&=\int_{t_n}^{t_{n+1}}\int_{t_n}^{t}\Delta v_n(s)\,\dd\beta(s)\,\dd t\\
\epsilon_n^{4,2}&=-\frac{\ii}{2}\int_{t_n}^{t_{n+1}}\int_{t_n}^{t}\Delta^2 v_n(s)\,\dd s\,\dd t.
\end{align*}
Set also $\displaystyle E_n^{4,1}=\sum_{k=0}^{n-1}\epsilon_k^{4,1}$ and 
$\displaystyle E_n^{4,2}=\sum_{k=0}^{n-1}\epsilon_k^{4,2}$. Using Minkowski's inequality, one then gets
\[
\left(\E\left[\norm{E_n^4}_{H^m}^p\right]\right)^{\frac1p}\le \left(\E\left[\norm{E_n^{4,1}}_{H^m}^p\right]\right)^{\frac1p}+\left(\E\left[\norm{E_n^{4,2}}_{H^m}^p\right]\right)^{\frac1p}.
\]
On the one hand, observe that applying the stochastic Fubini theorem gives the equality
\[
\epsilon_n^{4,1}=\int_{t_n}^{t_{n+1}}\int_{s}^{t_{n+1}}\Delta v_n(s)\,\dd t\,\dd\beta(s)=\int_{t_n}^{t_{n+1}}(t_{n+1}-s)\Delta v_n(s)\,\dd\beta(s).
\]
Introduce the (adapted) auxiliary process $\overline{v}$, such that $\overline{v}(s)=(t_{k+1}-s)v_k(s)$ for all $s\in[t_k,t_{k+1}]$ and $k=0,\ldots,N-1$. Then the error term $E_n^{4,1}$ is rewritten as the It\^o integral
\[
E_n^{4,1}=\int_{0}^{t_n}\Delta \overline{v}(s)\,\dd \beta(s),
\]
and applying the Burkholder--Davis--Gundy and H\"older inequalities, for all $p\ge 2$, 
one obtains 
\begin{align*}
\left(\E\left[\norm{E_n^{4,1}}_{H^m}^p\right]\right)^{\frac1p}&=\left(\E\left[\norm{\int_{0}^{t_n}\Delta \overline{v}(s)\,\dd\beta(s)}_{H^m}^{p}\right]\right)^{\frac1p}\\
&\le C_p(T)\left(\int_{0}^{t_n}\E\left[\norm{\Delta \overline{v}(s)}_{H^m}^p\right]\,\dd s\right)^{\frac1p}\\
&\le C_p(T)\left(\int_{0}^{t_n}\E\left[\norm{\overline{v}(s)}_{H^{m+2}}^p\right]\,\dd s\right)^{\frac1p}.
\end{align*}
By definition of $\overline{v}(s)$ and of $v_k(s)$, one obtains
\begin{align*}
\int_{0}^{t_n}\E\left[\norm{\overline{v}(s)}_{H^{m+2}}^p\right]\,\dd s&=\sum_{k=0}^{n-1}\int_{t_{k}}^{t_{k+1}}\E\left[\norm{(t_{k+1}-s)v_k(s)}_{H^{m+2}}^p\right]\,\dd s\\
&\le \tau^p\sum_{k=0}^{n-1}\int_{t_{k}}^{t_{k+1}}\E\left[\norm{S(0,s)\Psi_0(u(t_k))}_{H^{m+2}}^p\right]\,\dd s\\
&\le \tau^{p+1}\sum_{k=0}^{n-1}\E\left[\norm{\Psi_0(u(t_k))}_{H^{m+2}}^p\right]\\
&\le C(T,\norm{u_0}_{H^{m+2}})\tau^p,
\end{align*}
using the isometry property~\eqref{preservS}, the inequality~\eqref{eq:lemmaNL-bound} and the exact preservation of the $L^2$ norm~\eqref{preservL2-ex} as well as the almost sure bound~\eqref{borneHm-ex} for the $H^{m+2}$ norm of the exact solution, 
see Proposition~\ref{prop-ex}.

On the other hand, for the second term, using Minkowski's inequality and the definition of the auxiliary processes $v_k$, one obtains
\begin{align*}
\left(\E\left[\norm{E_n^{4,2}}_{H^m}^p\right]\right)^{\frac1p}&\le \sum_{k=0}^{n-1}
\left(\E\left[\norm{\epsilon_k^{4,2}}_{H^m}^p\right]\right)^{\frac1p}\\
&\le \frac12\sum_{k=0}^{n-1}\int_{t_k}^{t_{k+1}}\int_{t_k}^{t}\left(\E\left[\norm{v_k(s)}_{H^{m+4}}^p\right]\right)^{\frac1p}\,\dd s
\,\dd t\\
&\le C\tau^2\sum_{k=0}^{n-1}\left(\E\left[\norm{S(0,t_k)\Psi_0(u(t_k))}_{H^{m+4}}^{p}\right]\right)^{\frac1p}\\
&\le C(T,\norm{u_0}_{H^{m+4}})\tau, 
\end{align*}
using the isometry property~\eqref{preservS}, the inequality~\eqref{eq:lemmaNL-bound} and the almost sure bound~\eqref{borneHm-ex} for the $H^{m+4}$ norm of the exact solution.

Gathering the estimates, one obtains the following estimate for the fourth error term: for all $n=0,\ldots,N$,
$$
\left(\E\left[\norm{E_n^4}_{H^m}^p\right]\right)^{\frac1p}\le C_{m,p}(T,\norm{u_0}_{H^{m+4}})\tau.
$$

\noindent$\bullet$ It remains to deal with the fifth error term. Using a second-order Taylor expansion, one has, for $t\in[t_n,t_{n+1}]$,
\begin{align*}
\Psi_0(u(t))-\Psi_0(u(t_n))&=\Psi_0'(u(t_n)).\left(u(t)-u(t_n)\right)\\
&\quad+\int_{0}^{1}(1-\xi)\Psi_0''((1-\xi)u(t_n)+\xi u(t)).(u(t)-u(t_n),u(t)-u(t_n))\,\dd\xi\\
&=\Psi_0'(u(t_n)).\left(\left(S(t,t_n)-I\right)u(t_n)+\ii\int_{t_n}^{t}S(t,s)\Psi_0(u(s))\,\dd s\right)+R_n(t),
\end{align*}
using the mild formulation~\eqref{mild} for the exact solution, where one has defined the auxiliary quantity
$R_n(t)=\displaystyle\int_{0}^{1}(1-\xi)\Psi_0''((1-\xi)u(t_n)+\xi u(t)).(u(t)-u(t_n),u(t)-u(t_n))\,\dd\xi$.

For all $n=0,\ldots,N-1$, set 
\begin{align*}
\epsilon_n^{5,1}&=-\ii\int_{t_n}^{t_{n+1}}S(0,t_n)\Psi_0'(u(t_n)).\left(\left(S(t,t_n)-I\right)u(t_n)\right)\,\dd t\\
\epsilon_n^{5,2}&=\int_{t_n}^{t_{n+1}}S(0,t_n)\Psi_0'(u(t_n)).\left(\int_{t_n}^{t}S(t,s)\Psi_0(u(s))\,\dd s\right)  \,\dd t\\
\epsilon_n^{5,3}&=-\ii\int_{t_n}^{t_{n+1}}S(0,t_n)R_n(t)\,\dd t,
\end{align*}
and $\displaystyle E_n^{5,j}=\sum_{k=0}^{n-1}\epsilon_n^{5,j}$, $j=1,2,3$. Note that $\epsilon_n^{5}=\epsilon_n^{5,1}+\epsilon_n^{5,2}+\epsilon_n^{5,3}$, and $E_n^5=E_n^{5,1}+E_n^{5,2}+E_n^{5,3}$, for all $n=0,\ldots,N-1$, 
and Minkowski's inequality yields
\[
\left(\E\left[\norm{E_n^5}_{H^m}^p\right]\right)^{\frac1p}\le \left(\E\left[\norm{E_n^{5,1}}_{H^m}^p\right]\right)^{\frac1p}+\left(\E\left[\norm{E_n^{5,2}}_{H^m}^p\right]\right)^{\frac1p}+\left(\E\left[\norm{E_n^{5,3}}_{H^m}^p\right]\right)^{\frac1p}.
\]
It remains to obtain estimates for each of the three error terms in the right-hand side above.

\noindent$(i)$ To treat the first error terms $E_n^{5,1}$ and $\epsilon_n^{5,1}$, one follows the same strategy as for the error terms $E_n^{4,1}$ and $\epsilon_n^{4,1}$ above. Let us define auxiliary processes: for all $n\in\N_0$ and $t\in[t_n,t_{n+1}]$, set
\[
w_n(t)=S(t,t_n)u(t_n).
\]
For each $n\in\N_0$, the process $\left(w_n(t)\right)_{t\in[t_n,t_{n+1}]}$ is the solution of the linear stochastic evolution
\[
\dd w_n(t)=\ii\Delta w_n(t) \circ \dd\beta(t)=\ii\Delta w_{n}(t)\,\dd\beta(s)-\frac12\Delta^2w_{n}(t)\,\dd t, 
\]
with initial value $w_n(t_n)=u(t_n)$, see the definition~\eqref{defS} of the random propagator $S(t,s)$. The local error term $\epsilon_n^{5,1}$ is rewritten as follows in terms of the auxiliary process $w_n$: 
\begin{align*}
\epsilon_n^{5,1}&=-\ii\int_{t_n}^{t_{n+1}}S(0,t_{n})\Psi_0'(u(t_n)).\left(S(t,t_n)-I\right)u(t_n)\,\dd t\\
&=-\ii\int_{t_n}^{t_{n+1}}S(0,t_n)\Psi_0'(u(t_n)).\left(w_n(t)-w_n(t_n)\right)\,\dd t\\
&=\int_{t_n}^{t_{n+1}}S(0,t_n)\Psi_0'(u(t_n)).\int_{t_n}^{t}\Delta w_n(s)\,\dd\beta(s)\,\dd t\\
&+\frac{\ii}{2}\int_{t_n}^{t_{n+1}}S(0,t_n)\Psi_0'(u(t_n)).\int_{t_n}^{t}\Delta^2 w_n(s)\,\dd s\,\dd t\\
&=\epsilon_n^{5,1,1}+\epsilon_n^{5,1,2},
\end{align*}
with 
\begin{align*}
\epsilon_n^{5,1,1}&=\int_{t_n}^{t_{n+1}}S(0,t_n)\Psi_0'(u(t_n)).\int_{t_n}^{t}\Delta w_n(s)\,\dd\beta(s)\,\dd t\\
\epsilon_n^{5,1,2}&=\frac{\ii}{2}\int_{t_n}^{t_{n+1}}S(0,t_n)\Psi_0'(u(t_n)).\int_{t_n}^{t}\Delta^2 w_n(s)\,\dd s\,\dd t.
\end{align*}
Set also $E_n^{5,1,1}=\sum_{k=0}^{n-1}\epsilon_k^{5,1,1}$ and $E_n^{5,1,2}=\sum_{k=0}^{n-1}\epsilon_k^{5,1,2}$. Using Minkowski's inequality, one then gets
\[
\left(\E\left[\norm{E_n^{5,1}}_{H^m}^p\right]\right)^{\frac1p}\le \left(\E\left[\norm{E_n^{5,1,1}}_{H^m}^p\right]\right)^{\frac1p}+\left(\E\left[\norm{E_n^{5,1,2}}_{H^m}^p\right]\right)^{\frac1p}.
\]
On the one hand, observe that applying the stochastic Fubini theorem gives the equality
\begin{align*}
\epsilon_n^{5,1,1}&=\int_{t_n}^{t_{n+1}}S(0,t_n)\Psi_0'(u(t_n)).\int_{t_n}^{t}\Delta w_n(s)\,\dd\beta(s)\,\dd t\\
&=S(0,t_n)\Psi_0'(u(t_n)).\int_{t_n}^{t_{n+1}}\int_{s}^{t_{n+1}}\Delta w_n(s)\,\dd t\,\dd\beta(s)\\
&=S(0,t_n)\Psi_0'(u(t_n)).\int_{t_n}^{t_{n+1}}(t_{n+1}-s)\Delta w_n(s)\,\dd\beta(s).
\end{align*}
Introduce the (adapted) auxiliary process $\overline{w}$, such that 
$\overline{w}(s)=(t_{k+1}-s)S(0,t_k)\Psi'_0(u(t_k)).(\Delta w_k(s))$ 
for all $s\in[t_k,t_{k+1}]$ and $k=0,\ldots,N-1$. Then the error term $E_n^{5,1,1}$ is rewritten as the It\^o integral
\[
E_n^{5,1,1}=\int_{0}^{t_n}\overline{w}(s)\,\dd\beta(s),
\]
and applying the Burkholder--Davis--Gundy and H\"older's inequalities, 
for all $p\ge 2$, one obtains 
\begin{align*}
\left(\E\left[\norm{E_n^{5,1,1}}_{H^m}^p\right]\right)^{\frac1p}&=\left(\E\left[\norm{\int_{0}^{t_n}\overline{w}(s)\,\dd\beta(s)}_{H^m}^{p}\right]\right)^{\frac1p}\\
&\le C_p(T)\left(\int_{0}^{t_n}\E\left[\norm{\overline{w}(s)}_{H^{m}}^p\right]\,\dd s\right)^{\frac1p}.
\end{align*}
By definition of $\overline{w}(s)$ and of $w_k(s)$, one obtains
\begin{align*}
\int_{0}^{t_n}\E\left[\norm{\overline{w}(s)}_{H^{m}}^p\right]\,\dd s&=\sum_{k=0}^{n-1}\int_{t_{k}}^{t_{k+1}}\E\left[\norm{(t_{k+1}-s)S(0,t_k)\Psi'_0(u(t_k)).\Delta w_k(s)}_{H^{m}}^p\right]\,\dd s\\
&\le C\tau^p\sum_{k=0}^{n-1}\int_{t_{k}}^{t_{k+1}}\E\left[\norm{\Psi'_0(u(t_k)).\Delta S(s,t_k)u(t_k)}_{H^{m}}^p\right]
\,\dd s\\
&\le C\tau^{p+1}\sum_{k=0}^{n-1}\E\left[\norm{u(t_k)}_{H^{m}}^{2p}\norm{\Delta u(t_k)}_{H^{m}}^p\right]\\
&\le C\tau^{p+1}\sum_{k=0}^{n-1}\E\left[\norm{u(t_k)}_{H^{m+2}}^{3p}\right]\\
&\le C(T,\norm{u_0}_{H^{m+2}})\tau^p,
\end{align*}
using the isometry property~\eqref{preservS}, the inequality~\eqref{eq:lemmaNL-d1} and the almost sure bound~\eqref{borneHm-ex} for 
the $H^{m+2}$ norm of the exact solution.

On the other hand, for the second term, using Minkowski's inequality and the definition of the auxiliary processes $w_k$, one obtains
\begin{align*}
\left(\E\left[\norm{E_n^{5,1,2}}_{H^m}^p\right]\right)^{\frac1p}&\le \sum_{k=0}^{n-1}
\left(\E\left[\norm{\epsilon_k^{5,1,2}}_{H^m}^p\right]\right)^{\frac1p}\\
&\le \frac12\sum_{k=0}^{n-1}\int_{t_k}^{t_{k+1}}\int_{t_k}^{t}\left(\E\left[\norm{S(0,t_k)\Psi'_0(u(t_k)).\Delta^2 S(s,t_k)u(t_k)}_{H^{m}}^p\right]\right)^{\frac1p}\,\dd s
\,\dd t\\
&\le C\tau^{2}\sum_{k=0}^{n-1}\left(\E\left[\norm{u(t_k)}_{H^m}^{2p}\norm{\Delta^2 u(t_k)}_{H^{m}}^{p}\right]\right)^{\frac1p}\\
&\le C\tau^2\sum_{k=0}^{n-1}\left(\E\left[\norm{u(t_k)}_{H^{m+4}}^{3p}\right]\right)^{\frac1p}\\
&\le C(T,\norm{u_0}_{H^{m+4}})\tau, 
\end{align*}
using the isometry property~\eqref{preservS}, the inequality~\eqref{eq:lemmaNL-bound} and the almost sure bound~\eqref{borneHm-ex} for the $H^{m+4}$ norm of the exact solution.

Gathering the estimates for $E_n^{5,1,1}$ and $E_n^{5,1,2}$, one finally obtains, for all $n=0,\ldots,N$,
\[
\left(\E\left[\norm{E_n^{5,1}}_{H^m}^p\right]\right)^{\frac1p}\le C_{m,p}(T,\norm{u_0}_{H^{m+4}})\tau.
\]

\noindent$(ii)$ To deal with the error terms $E_n^{5,2}$ and $\epsilon_n^{5,2}$, using Minkowski's inequality, the isometry property~\eqref{preservS} and the inequalities~\eqref{eq:lemmaNL-d1} and~\eqref{eq:lemmaNL-bound} (Lemma~\ref{lemmaNL}), one obtains
\begin{align*}
\left(\E\left[\norm{E_n^{5,2}}_{H^m}^p\right]\right)^{\frac1p}&\le \sum_{k=0}^{n-1}\left(\E\left[\norm{\epsilon_k^{5,2}}_{H^m}^p\right]\right)^{\frac1p}\\
&\le \sum_{k=0}^{n-1}\int_{t_k}^{t_{k+1}}\int_{t_k}^{t}
\left(\E\left[\norm{S(0,t_k)\Psi_0'(u(t_k)).\left(S(t,s)\Psi_0(u(s))\right)}_{H^m}^p\right]\right)^{\frac1p}\,\dd s\,\dd t\\
&\le \sum_{k=0}^{n-1}\int_{t_k}^{t_{k+1}}\int_{t_k}^{t}
\left(\E\left[\norm{u(t_k)}_{H^m}^{2p}\norm{\Psi_0(u(s))}_{H^m}^p\right]\right)^{\frac1p}\,\dd s\,\dd t\\
&\le C\tau \sum_{k=0}^{n-1}\int_{t_k}^{t_{k+1}}
\left(\E\left[\norm{u(t_k)}_{H^m}^{2p}\norm{u(s)}_{L^2}^{2p}\norm{u(s)}_{H^m}^p\right]\right)^{\frac1p}\,\dd s.
\end{align*}
Finally using the almost sure bound~\eqref{borneHm-ex}
for the $H^m$ norm of the exact solution, one obtains, for all $n=0,\ldots,N$
\[
\left(\E\left[\norm{E_n^{5,2}}_{H^m}^p\right]\right)^{\frac1p}\le C_{m,p}(T,\norm{u_0}_{H^m})\tau.
\]

\noindent$(iii)$ To deal with the error terms $E_n^{5,3}$ and $\epsilon_n^{5,3}$, using Minkowski's inequality gives 
\begin{align*}
\left(\E\left[\norm{E_n^{5,3}}_{H^m}^p\right]\right)^{\frac1p}&\le \sum_{k=0}^{n-1}\left(\E\left[\norm{\epsilon_k^{5,3}}_{H^m}^p\right]\right)^{\frac1p}\\
&\le \sum_{k=0}^{n-1}\int_{t_k}^{t_{k+1}}\left(\E\left[\norm{S(0,t_k)R_k(t)}_{H^m}^p\right]\right)^{\frac1p}\,\dd t.
\end{align*}
Using the isometry property~\eqref{preservS}, the inequality~\eqref{eq:lemmaNL-d2} (Lemma~\ref{lemmaNL}), one obtains
\begin{align*}
\left(\E\left[\norm{E_n^{5,3}}_{H^m}^p\right]\right)^{\frac1p}&\le C\sum_{k=0}^{n-1}\int_{t_k}^{t_{k+1}}
\left(\E\left[\left(\norm{u(t)}_{H^m}+\norm{u(t_k)}_{H^m}\right)^{p}\norm{u(t)-u(t_k)}_{H^m}^{2p}\right]\right)^{\frac1p}\,\dd t\\
&\le C_{p}(T,\norm{u_0}_{H^m})\sum_{k=0}^{n-1}\int_{t_k}^{t_{k+1}}\left(\E\left[\norm{u(t)-u(t_k)}_{H^m}^{2p}\right]\right)^{\frac1p}\,\dd t\\
&\le C_{m,p}(T,\norm{u_0}_{H^{m+2}})\tau,
\end{align*}
where the almost sure bound~\eqref{borneHm-ex} for the $H^m$ norm of the exact solution and the temporal regularity estimate~\eqref{tempregul} (Proposition~\ref{prop-ex}) have been used.

Gathering the estimates, one finally obtains the last required result: for all $n=0,\ldots,N$
$$
\left(\E\left[\norm{E_n^5}_{H^m}^p\right]\right)^{\frac1p}\le C_{m,p}(T,\norm{u_0}_{H^{m+4}})\tau.
$$

\noindent$\bullet$ We are now in position to obtain the error estimate~\eqref{error}. Gathering the previously obtained estimates, one has, for all $n=0,\ldots,N$
\begin{align*}
\left(\E\left[\norm{E_n}_{H^m}^p\right]\right)^{\frac1p}&\le \sum_{j=1}^5\left(\E\left[\norm{E_n^j}_{H^m}^p\right]\right)^{\frac1p}\\
&\le C_{m,p}(T,\norm{u_0}_{H^{m+4}})\tau+
C(T,\norm{u_0}_{H^m})\tau\sum_{k=0}^{n-1}\left(\E\left[\norm{E_k}_{H^m}^p\right]\right)^{\frac1p}.
\end{align*}
Applying the discrete Gronwall Lemma then gives~\eqref{error} and concludes the proof of Theorem~\ref{th-ms}.

\end{proof}

\section{Numerical experiments}\label{sec-numexp}
We present some numerical experiments in order to support and illustrate 
the above theoretical results. In addition, we shall compare 
the behavior of the splitting scheme \eqref{split} 
(denoted by \textsc{Split} below) with the following time integrators for the stochastic nonlinear Schr\"odinger equation \eqref{prob}
\begin{itemize}
\item the stochastic exponential integrator from \cite{MR3736655} 
(adapted to the present nonlocal interaction cubic nonlinearity, denoted by \textsc{Exp})
$$
u_{n+1}=S(t_{n+1},t_n)u_n+\ii \tau S(t_{n+1},t_n)V[u_n]u_n.
$$
\item the semi-implicit midpoint scheme (denoted \textsc{Mid})
$$
\ii\frac{u_{n+1}-u_n}{\tau}+\frac{\chi_n}{\sqrt{\tau}}\Delta u_{n+1/2}+V[u_n]u_n=0,
$$
where $u_{n+1/2}=\frac12\left(u_n+u_{n+1}\right)$ and $\chi_n=\frac{\beta(t_{n+1})-\beta(t_n)}{\sqrt{\tau}}$. 
This is a modification of the Crank--Nicolson from \cite{MR3312594} for the nonlinear interaction nonlinearity studied here. 
\end{itemize}
Unless stated otherwise, we consider the SPDE \eqref{prob} with the potential $V(x)=\cos(x)$ 
on the one dimensional torus with periodic boundary conditions $[0,2\pi]$. 
The spatial discretization is done by a pseudospectral method with
$M$ Fourier modes. The initial value is given by $u_0(x)=\exp(-0.5(x-\pi)^2)$.

\subsection{Evolution plots}
To illustrate the interplay and the balance between the random dispersion 
and the nonlinearity, in Figure~\ref{fig:evo}, we display the evolution of $|u_n|^2$ 
along one sample of the numerical solution obtained by the splitting integrator \eqref{split}.
The discretization parameters are $\tau=2^{-14}$ and $M=2^{10}$ and the time interval is given by $[0,0.5]$.

\begin{figure}
\centering
\begin{subfigure}[b]{0.4\textwidth}
\hspace{-1cm}\includegraphics[height=5cm,keepaspectratio]{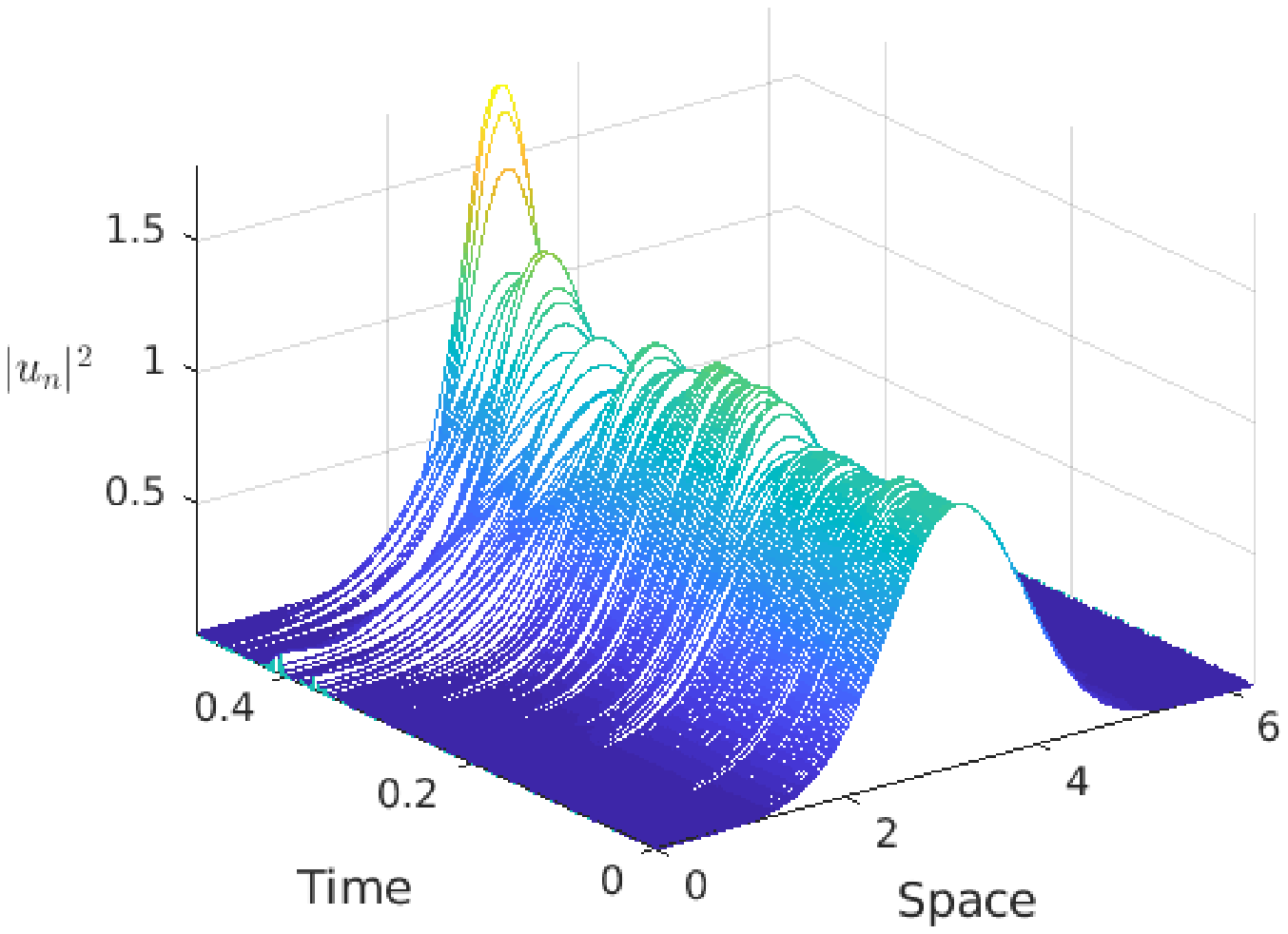}
\caption{Space-time evolution}
\end{subfigure}
\begin{subfigure}[b]{0.4\textwidth}
\includegraphics[height=5cm,keepaspectratio]{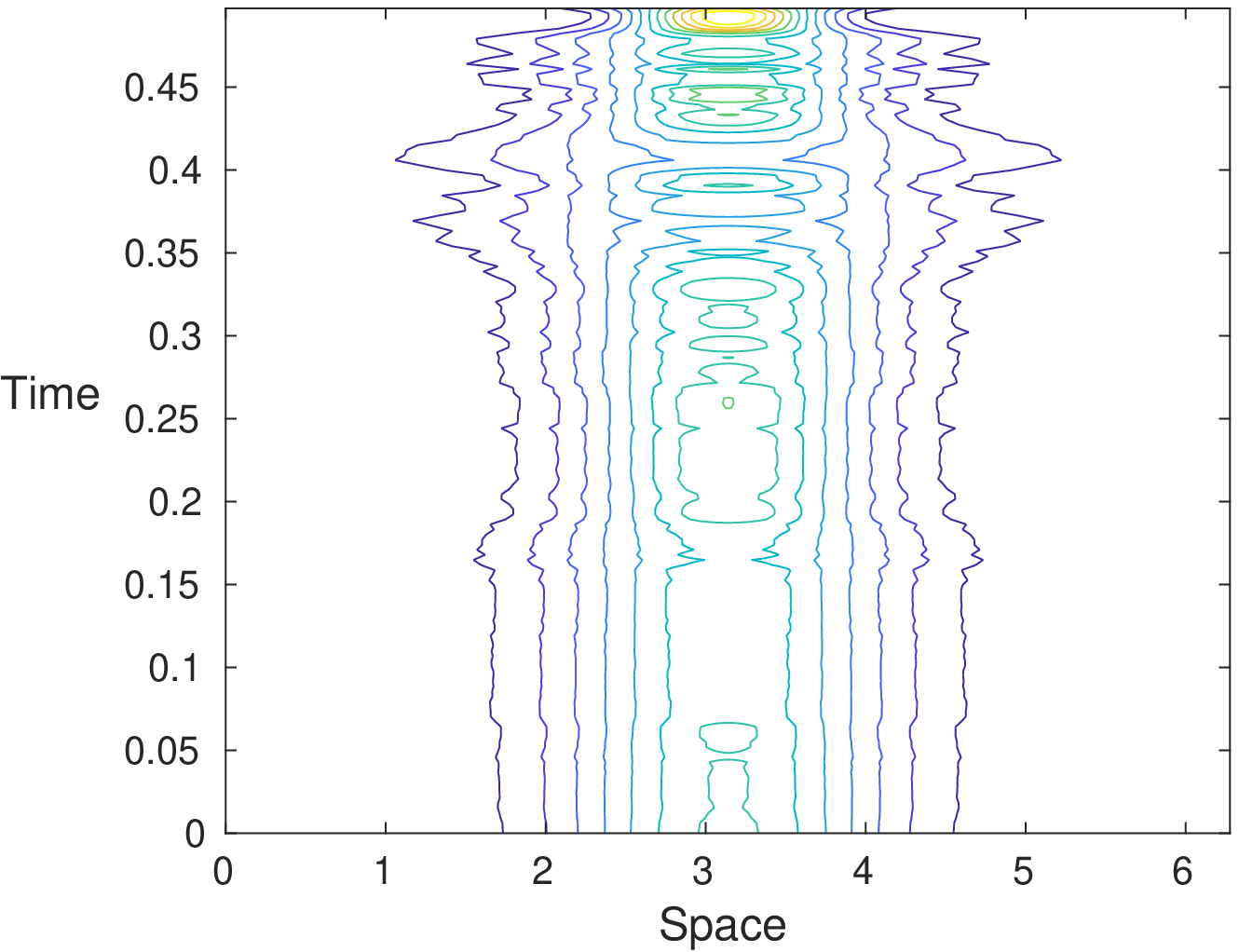}
\caption{Contour plot}
\end{subfigure}
\caption{Space-time evolution and contour plot for the splitting integrator \eqref{split}. 
The discretization parameters are $\tau=2^{-14}$ and $M=2^{10}$ Fourier modes.}
\label{fig:evo}
\end{figure}

\subsection{Conservation of the $\mathbf{L}^{\mathbf2}$-norm}
It is known that the $L^2$-norm of the solution to the SPDE \eqref{prob} 
remains constant for all times, see Proposition~\ref{prop-ex}. Figure~\ref{fig:massconvol} illustrates the corresponding
behavior of the above numerical integrators along one sample path. 
For this numerical experiment, we consider the parameters 
$\tau=2^{-8}$ and $M=2^{10}$ Fourier modes and the time interval $[0,1]$. 
Exact preservation of the $L^2$-norm for the splitting scheme is observed, as stated in Proposition~\ref{prop-norm2}, 
whereas a small drift is observed for the exponential integrator and the midpoint scheme. 

\begin{figure}
\begin{center}
\includegraphics*[height=6cm,keepaspectratio]{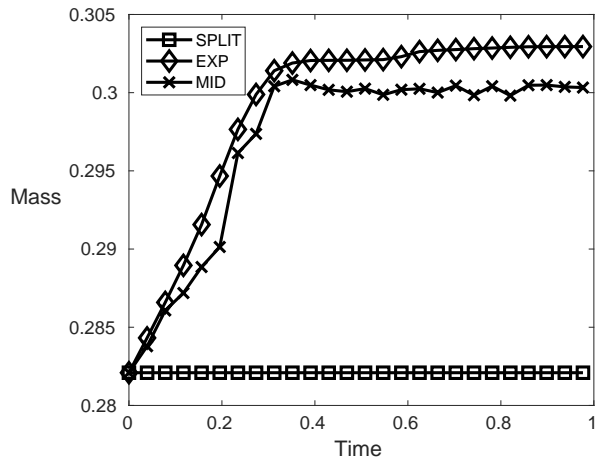}
\caption{Preservation of the $L^2$-norm: Splitting scheme ($\square$), 
exponential integrator ($\Diamond$), and midpoint scheme ($\times$).}
\label{fig:massconvol}
\end{center}
\end{figure}

\subsection{Strong convergence}
We now illustrate the mean-square convergence of 
the splitting scheme \eqref{split} stated in Theorem~\ref{th-ms}.
$M=2^{10}$ Fourier modes are used 
for the spatial discretization. 
The mean-square errors $\E[\norm{u_{\text{ref}}(x,T_{\text{end}})-u_N(x)}_{H^1}^2]^{1/2}$
at time $T_{\text{end}}=1$ are displayed in Figure~\ref{fig:params} 
for various values of the time step $\tau=2^{-\ell}$ for $\ell=10,\ldots,16$. 
Here, we simulate the reference solution $u_{\text{ref}}(x,t)$ 
with the splitting scheme, with a small time step $\tau_{\text{ref}}=2^{-18}$. 
The expected values are approximated by computing averages 
over $M_s=100$ samples. 
In Figure~\ref{fig:params}, we observe convergence of 
order $1$ for all time integrators. Note that, the strong order of convergence of 
the exponential scheme and midpoint integrator are not known in the case 
of the considered nonlocal interaction potential, 
whereas Figure~\ref{fig:params} illustrates our main result Theorem~\ref{th-ms} for the splitting scheme. 

\begin{figure}
\centering
\includegraphics[height=5cm,keepaspectratio]{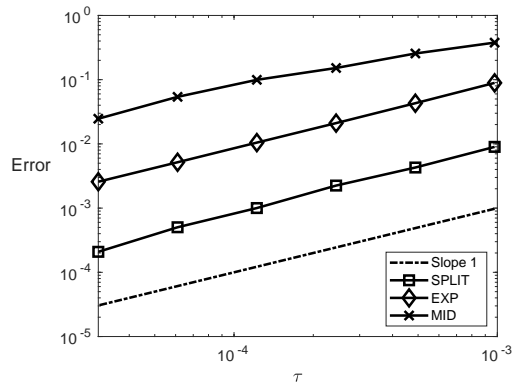}
\caption{Mean-square errors as a function of the time step: Splitting scheme ($\square$), 
exponential integrator ($\Diamond$), and midpoint scheme ($\times$). 
The dotted line has slope $1$.}
\label{fig:params}
\end{figure}

\section{Acknowledgements}
The work of CEB was partially supported by the SIMALIN project ANR-19-CE40-0016 of the French National Research Agency.
The work of DC was partially supported by the Swedish Research Council (VR) (projects nr. $2018-04443$). 
The computations were performed on resources provided by the Swedish National Infrastructure 
for Computing (SNIC) at HPC2N, Ume{\aa} University and at 
Chalmers Centre for Computational Science and Engineering. 


\end{document}